\documentclass[12pt]{article}

\usepackage{amsmath,amssymb,amsfonts}
\usepackage{amsthm}
\usepackage{mathrsfs}
\usepackage{mathtools}
\usepackage{bm}
\usepackage{enumitem}
\usepackage{xcolor}
\usepackage{microtype,hyperref}
\usepackage[margin=1in]{geometry}

\usepackage[numbers,sort&compress]{natbib}

\theoremstyle{thmstyleone}
\newtheorem{theorem}{Theorem}[section]
\newtheorem{proposition}[theorem]{Proposition}
\newtheorem{lemma}[theorem]{Lemma}
\newtheorem{corollary}[theorem]{Corollary}

\theoremstyle{thmstyletwo}
\newtheorem{remark}[theorem]{Remark}
\newtheorem{example}[theorem]{Example}

\numberwithin{equation}{section}

\theoremstyle{thmstylethree}
\newtheorem{definition}[theorem]{Definition}

\newcommand{\R}{\mathbb R}
\newcommand{\Acal}{\mathcal A}
\newcommand{\Dcal}{\mathfrak D}

\begin{document}

\begin{center} 
\MakeUppercase{\bf \Large A fractional de Rham complex for} \\[.2cm]
\MakeUppercase{\bf \Large coframe-attached Maxwell equations}
\end{center}

\begin{center}
{\sc Marvin Fritz} \\[.2cm]
\small Faculty of Mathematics, University of Vienna, Vienna, Austria
\end{center}

\begin{quote}\small
\textsc{Abstract.} We develop a coordinate-anchored pointwise theory of fractional tangent functionals on bounded rectangles. By working on the range of the coordinatewise Riemann--Liouville integral, we define anchored fractional operators by exact inversion and prove a representation theorem: every coordinate fractional tangent functional is a scalar multiple of evaluation of the inverse operator at the base point, and the normalized one is unique. At interior points, the resulting fractional tangent space acts faithfully on the common anchored space.
We then construct an exterior algebra over the polynomial algebra generated by the fractional coordinate primitives. Its fractional exterior differential defines a nilpotent algebraic chain complex, admits an explicit polynomial Poincar\'e homotopy, and is related to the ordinary polynomial de~Rham complex by a positive diagonal rescaling. Since this fractional differential is not a graded derivation for the ordinary wedge product, the global object is a de~Rham-type chain complex. We also prove a rigidity theorem for positive-cone-preserving linear coordinate changes preserving the fractional coframe and formulate a Lorentzian coframe-attached Maxwell-type system, deriving charge conservation and fractional wave equations within the polynomial coefficient class.
\end{quote}

\section{Introduction}

Classical differential geometry builds its tangent--cotangent formalism from
local derivations: evaluation of a coordinate partial derivative at a point is
the prototypical tangent vector, and differential forms are dual to these
operators. In fractional analysis the corresponding question is less direct.
Caputo and Riemann--Liouville operators are nonlocal and are central in models
with memory effects, including anomalous diffusion, viscoelasticity, control and
mathematical physics, see the seminal works
\cite{Caputo1967,MetzlerKlafter2000,BagleyTorvik1983,Engheta1997} and the monographs \cite{OldhamSpanier1974,Podlubny1999,Diethelm2010,SamkoKilbasMarichev1993,KilbasSrivastavaTrujillo2006,Mainardi2010}.
Caputo-in-time PDEs also arise in applied and continuum models such as
subdiffusive tumor growth, fractional Fokker--Planck dynamics, polymeric fluids,
and time-fractional phase-field systems
\cite{fritz2021subdiffusive,fritz2024analysis,fritz2024well,fritz2026time}.
However, they do not satisfy the classical Leibniz rule in the ordinary form, so
one cannot obtain a satisfactory fractional tangent space by simply replacing
\(\partial_i\) with a Caputo derivative in the usual definition of derivation.
The aim of this paper is to construct a controlled pointwise substitute and to
identify exactly what global exterior calculus it supports.

{Fractional exterior calculi have been studied before. Cottrill-Shepherd and
Naber introduced fractional exterior derivatives, form spaces, closedness and
exactness, and metric structures
\cite{CottrillShepherdNaber2001,CottrillShepherdNaber2003}; Tarasov used
fractional differential forms and exterior derivatives in the study of
fractional gradient and Hamiltonian systems and developed related vector- and
exterior-calculus formulations
\cite{Tarasov2005GradientHamiltonian,Tarasov2008,Tarasov2010Exterior}.
The distinction is that the present
construction starts from local anchored function spaces and from exact inversion
of the coordinatewise Riemann--Liouville integral. Thus the pointwise fractional
tangent functionals are represented before any formal exterior algebra is
introduced. The global exterior complex is then built only on a polynomial
coefficient class where closure, mixed commutation, nilpotency, and the
Poincar\'e homotopy can be proved directly. Consequently, the Maxwell-type
system in Section~7 is a coframe-attached algebraic model rather than an
intrinsic coordinate-free fractional electromagnetic theory.}

The present work lies at the intersection of several strands: the choice of
domains for fractional differential and integral operators, Caputo-type
fractional PDE models in continuum physics, nonclassical variational calculi, and
homotopy-based methods for exterior differential systems
\cite{KleinerHilfer2021,KhaiderAzanzalAbderrahmaneSaid2024,GolmankhanehTuncDepollierZayed2025,KyciaSilhan2025}.
In comparison with these works, our focus is not on a specific evolution
equation or variational principle, but on the anchored local linear structure
behind a Caputo-type tangent--cotangent formalism and on the formal de~Rham
complex it induces.

The basic normalization is
\begin{equation}
D_i(I_i^\alpha h)=h(p),
\end{equation}
where \(I_i^\alpha\) is the left Riemann--Liouville integral in the \(i\)-th
coordinate and \(p\in U\) is fixed. Although this is the correct normalization
for evaluation of the anchored Caputo derivative, the normalized class is not
closed under addition. We therefore first define the vector space of all scalar
multiples of such functionals and select the normalized generator only afterward.
The resulting representation theorem states that
\begin{equation}
\Dcal_i^\alpha(p)=
\R\cdot \bigl((\partial_{C,i}^\alpha)|_p\bigr),
\end{equation}
and, at interior base points, the direct sum
\(T_p^\alpha U=\bigoplus_i\Dcal_i^\alpha(p)\) acts faithfully on the common
anchored space \(\Acal^\alpha(U)=\bigcap_i\Acal_i^\alpha(U)\).
 The pointwise dual basis
suggests symbols \(d^\alpha X_i\), and we build an exterior algebra over the
polynomial ring generated by the fractional coordinate primitives
\(X_i^\alpha=x_i^\alpha/\Gamma(1+\alpha)\). On this coefficient class the
fractional exterior differential is stable and nilpotent. We also prove an
explicit fractional polynomial Poincar\'e homotopy and a rigidity theorem showing
that, among positive-cone-preserving linear coordinate changes for which the
fractional powers are real-valued on the anchored box, preservation of the
fractional coordinate coframe forces coordinate permutations and positive
rescalings. At the same time, an explicit
positive diagonal rescaling identifies the formal complex with the ordinary
polynomial de~Rham complex. Thus the global complex is a chain-level realization
of the anchored local theory.

The mathematical-physics motivation is provided by Maxwell's equations in
differential-form language \cite{Flanders1989}. Fractional electromagnetic
models have appeared in work on multipole expansions, fractional curl operators,
and fractional Maxwell systems
\cite{Engheta1996,Engheta1997,Engheta1998,Tarasov2008,BaleanuMaxwell}. In the
present setting we formulate a Lorentzian coframe-attached Maxwell-type system
on a bounded spacetime box. Nilpotency gives fractional charge conservation, and
the component equations imply fractional wave equations for the electric and
magnetic fields. The Hodge-type operator used there is attached to the chosen
fractional coframe.

The paper is organized as follows. Section~2 introduces the coordinatewise
anchored spaces and the inverse Caputo operators. Sections~3 and~4 prove the
representation theorem, define the fractional tangent space, and establish
faithfulness at interior base points. Sections~5 and~6 construct the formal
fractional exterior complex, prove nilpotency, the Poincar\'e homotopy and the
coordinate-rigidity result. Section~7 develops the Lorentzian coframe-attached
Maxwell-type system and derives the continuity and wave equations.

\section{Preliminaries}

The purpose of this section is twofold. First, we fix the coordinatewise fractional operators that will be used throughout the paper. Second, and more importantly, we introduce the anchored function spaces on which the local fractional derivation theory becomes exact rather than merely formal. The key point is that on the anchored spaces each function is, in every coordinate direction separately, an \(\alpha\)-fractional primitive of a continuous function. This gives a coordinatewise reconstruction formula and allows us to define the relevant Caputo operator through inversion rather than through an a priori singular integral formula.

Throughout the paper we fix a real number \(\alpha\in(0,1)\), a rectangle
\begin{equation}
U:=\prod_{i=1}^d [0,b_i]\subset \R^d,
\qquad b_i>0,
\end{equation}
and a base point \(p=(p_1,\dots,p_d)\in U\).

\subsection{Coordinatewise fractional operators}

{Throughout this subsection the fractional derivative order is the fixed number
\(\alpha\in(0,1)\). We allow the auxiliary Riemann--Liouville integral parameter
\(\beta>0\) in \eqref{eq:coordinate-RL-integral}, because the semigroup property
and the Caputo representation use both \(\beta=\alpha\) and
\(\beta=1-\alpha\). Whenever the fractional integral is used as the primitive
associated with the order-\(\alpha\) derivative, the relevant choice is
\(\beta=\alpha\).}

{For \(\beta>0\) and a continuous function \(f:U\to\R\), define the left
Riemann--Liouville integral in the \(i\)-th variable by
\begin{equation}\label{eq:coordinate-RL-integral}
(I_i^\beta f)(x)
:= \frac{1}{\Gamma(\beta)}\int_0^{x_i}(x_i-\xi)^{\beta-1}
 f(x_1,\dots,x_{i-1},\xi,x_{i+1},\dots,x_d)\,d\xi.
\end{equation}
For a function \(a:U\to\R\) that is continuously differentiable in the variable
\(x_i\), define the classical partial Caputo derivative of order
\(\alpha\in(0,1)\) by
\begin{equation}\label{eq:coordinate-Caputo-derivative}
(\partial_i^\alpha a)(x)
:= \frac{1}{\Gamma(1-\alpha)}
\int_0^{x_i}(x_i-\xi)^{-\alpha}
\partial_i a(x_1,\dots,x_{i-1},\xi,x_{i+1},\dots,x_d)\,d\xi.
\end{equation}}

We will use the standard semigroup property
\begin{equation}
I_i^\mu I_i^\nu = I_i^{\mu+\nu},
\qquad \mu,\nu>0,
\end{equation}
which holds pointwise for continuous functions; see, for example, \cite[Chapter~2]{Podlubny1999}, \cite[Chapter~2]{SamkoKilbasMarichev1993}, and \cite[Section~2.6]{KilbasSrivastavaTrujillo2006}.

\begin{remark}\label{rem:no-wrong-identity}
The two operators
\begin{equation}
\partial_i^\alpha = I_i^{1-\alpha}\partial_i,
\qquad
{}^{RL}\partial_i^\alpha = \partial_i I_i^{1-\alpha},
\end{equation}
both reduce to \(\partial_i\) as \(\alpha\to 1^-\) but differ on functions that do not vanish at \(\{x_i=0\}\); the identity \(\partial_i^\alpha=\partial_i I_i^{1-\alpha}\) is in particular \emph{not} valid in general. See, for example, \cite[Chapter~2]{OldhamSpanier1974}, \cite[Chapter~2]{Podlubny1999}, and \cite[Sections~2.4--2.5]{KilbasSrivastavaTrujillo2006}. In the present paper, we avoid this ambiguity by defining the anchored derivative below through the exact inverse relation on \(\Acal_i^\alpha(U)\).
\end{remark}

\subsection{Coordinatewise anchored spaces}

The right function spaces for our purposes are not all smooth functions, but the coordinatewise ranges of the fractional integral operators.

\begin{definition}[coordinatewise anchored spaces]\label{def:Ai}
For each \(i\in\{1,\dots,d\}\), define
\begin{equation}
\Acal_i^\alpha(U)
:=
\Big\{
f\in C(U): \exists h_i\in C(U)\text{ such that } f=I_i^\alpha h_i
\Big\}.
\end{equation}
We also define the common anchored space
\begin{equation}
\Acal^\alpha(U):=\bigcap_{i=1}^d \Acal_i^\alpha(U).
\end{equation}
\end{definition}

Thus \(f\in \Acal_i^\alpha(U)\) vanishes on the coordinate hyperplane \(\{x_i=0\}\), while \(f\in\Acal^\alpha(U)\) vanishes on every coordinate hyperplane. The space \(\Acal^\alpha(U)\) is the natural common domain on which all coordinatewise anchored Caputo derivatives are simultaneously available.

\begin{proposition}\label{prop:Ai-vector-space}
For each \(i\), the space \(\Acal_i^\alpha(U)\) is a real vector space. Consequently \(\Acal^\alpha(U)\) is a real vector space as well.
\end{proposition}

\begin{proof}
Fix \(i\). Let \(f,g\in \Acal_i^\alpha(U)\) and \(\lambda\in\R\). Choose \(h_i,k_i\in C(U)\) such that
\begin{equation}
f=I_i^\alpha h_i,
\qquad
g=I_i^\alpha k_i.
\end{equation}
Then
\begin{equation}
f+\lambda g = I_i^\alpha(h_i+\lambda k_i),
\end{equation}
and \(h_i+\lambda k_i\in C(U)\). Hence \(f+\lambda g\in \Acal_i^\alpha(U)\). Since \(\Acal^\alpha(U)\) is an intersection of vector spaces, it is itself a vector space.
\end{proof}

\begin{lemma}[coordinatewise anchored inversion]\label{lem:Ai-inversion}
Let \(f\in \Acal_i^\alpha(U)\). Then there exists a unique \(h_i\in C(U)\) such that
\begin{equation}
f=I_i^\alpha h_i.
\end{equation}
\end{lemma}

\begin{proof}
Existence is part of Definition~\ref{def:Ai}. For uniqueness, suppose
\begin{equation}
I_i^\alpha h_i = I_i^\alpha \tilde h_i.
\end{equation}
Fix the variables \((x_1,\dots,x_{i-1},x_{i+1},\dots,x_d)\) and regard both sides as functions of \(x_i\in[0,b_i]\). The standard injectivity of the left Riemann--Liouville integral on \(C[0,b_i]\) implies \(h_i=\tilde h_i\) pointwise; see \cite[Chapter~2]{Podlubny1999}, \cite[Chapter~2]{SamkoKilbasMarichev1993}, or \cite[Section~2.1]{KilbasSrivastavaTrujillo2006}.
\end{proof}

\begin{definition}[coordinatewise anchored Caputo derivative]\label{def:anchored-caputo-new}
Fix \(i\in\{1,\dots,d\}\). For \(f\in \Acal_i^\alpha(U)\), let \(h_i\in C(U)\) be the unique function from Lemma~\ref{lem:Ai-inversion} such that
\begin{equation}
f=I_i^\alpha h_i.
\end{equation}
We define the coordinatewise anchored Caputo derivative by
\begin{equation}
(\partial_{C,i}^\alpha f)(x):=h_i(x),
\end{equation}
and thus
\(
\partial_{C,i}^\alpha:\Acal_i^\alpha(U)\to C(U).
\)
\end{definition}

\begin{proposition}[linearity and reconstruction]\label{prop:reconstruction-new}
For each \(i\), the operator
\begin{equation}
\partial_{C,i}^\alpha:\Acal_i^\alpha(U)\to C(U)
\end{equation}
is linear, and every \(f\in\Acal_i^\alpha(U)\) satisfies
\begin{equation}
f = I_i^\alpha(\partial_{C,i}^\alpha f).
\end{equation}
\end{proposition}

\begin{proof}
Linearity follows from Proposition~\ref{prop:Ai-vector-space} and uniqueness in Lemma~\ref{lem:Ai-inversion}. The reconstruction formula is exactly the defining property of \(\partial_{C,i}^\alpha f\).
\end{proof}

\begin{proposition}[agreement with the classical Caputo derivative]\label{prop:anchored-vs-classical-new}
Fix \(i\in\{1,\dots,d\}\). Let \(f\in C(U)\) be continuously differentiable in the variable \(x_i\), with
\(\partial_i f\in C(U)\), and assume that
\begin{equation}
f(x_1,\dots,x_{i-1},0,x_{i+1},\dots,x_d)=0.
\end{equation}
Assume also that the classical partial Caputo derivative
\begin{equation}
(\partial_i^\alpha f)(x)
:=\frac{1}{\Gamma(1-\alpha)}
\int_0^{x_i}(x_i-\xi)^{-\alpha}\,\partial_i f(x_1,\dots,x_{i-1},\xi,x_{i+1},\dots,x_d)\,d\xi
\end{equation}
exists and belongs to \(C(U)\). Then \(f\in\Acal_i^\alpha(U)\), and
\begin{equation}
\partial_{C,i}^\alpha f = \partial_i^\alpha f.
\end{equation}
\end{proposition}

\begin{proof}
Because \(f\) is \(C^1\) in the variable \(x_i\) and vanishes on the anchoring hyperplane \(\{x_i=0\}\), the fundamental theorem of calculus in that variable gives
\begin{equation}
f(x)=\int_0^{x_i}\partial_i f(x_1,\dots,x_{i-1},\xi,x_{i+1},\dots,x_d)\,d\xi
= I_i^1(\partial_i f)(x).
\end{equation}
By the definition of the classical Caputo derivative,
\begin{equation}
\partial_i^\alpha f=I_i^{1-\alpha}(\partial_i f).
\end{equation}
The assumed continuity of \(\partial_i^\alpha f\) therefore implies
\begin{equation}
f=I_i^1(\partial_i f)=I_i^\alpha\bigl(I_i^{1-\alpha}\partial_i f\bigr)
=I_i^\alpha(\partial_i^\alpha f).
\end{equation}
Hence \(f\in\Acal_i^\alpha(U)\). Since \(\partial_i^\alpha f\in C(U)\), uniqueness in Lemma~\ref{lem:Ai-inversion} shows that \(\partial_i^\alpha f\) is the unique continuous function whose \(I_i^\alpha\)-integral is \(f\). Therefore it coincides with \(\partial_{C,i}^\alpha f\).
\end{proof}

For each \(i\), we define
\begin{equation}
X_i^\alpha(x):=\frac{x_i^\alpha}{\Gamma(1+\alpha)}.
\end{equation}
Then \(X_i^\alpha=I_i^\alpha 1\), so \(X_i^\alpha\in \Acal_i^\alpha(U)\). However, for \(d\ge2\), in general
\begin{equation}\label{rem:Xi-membership}
X_i^\alpha\notin \Acal^\alpha(U),
\end{equation}
because \(X_i^\alpha\) does not vanish on the hyperplanes \(\{x_j=0\}\) for \(j\neq i\).

\section{Coordinatewise fractional derivations at a point}

This section contains the algebraic core of the paper. A natural normalization is
\begin{equation}
D_i(I_i^\alpha h)=h(p).
\end{equation}
However, this condition is not preserved under addition. It is therefore unsuitable as the defining property of a vector space of derivations. The correct linear framework is to define the ambient vector space first and to select the normalized operators only afterward.

In particular, the objects introduced in this section are not derivations on a full function algebra in the classical differential-geometric sense. Rather, they are coordinatewise anchored linear functionals singled out by the reconstruction property associated with the operator \(I_i^\alpha\). No Leibniz rule is assumed in the definition, and no Leibniz rule is used in any of the constructions that follow.

\begin{definition}[\(i\)-th fractional derivation]\label{def:frac-deriv-new}
Fix \(i\in\{1,\dots,d\}\). An \emph{\(i\)-th fractional derivation of order \(\alpha\) at \(p\)} is a linear map
\begin{equation}
D_i:\Acal_i^\alpha(U)\to\R
\end{equation}
for which there exists a scalar \(\lambda_i(D_i)\in\R\) such that
\begin{equation}
D_i(I_i^\alpha h)=\lambda_i(D_i)\,h(p)
\qquad\text{for every } h\in C(U).
\end{equation}
We write \(\Dcal_i^\alpha(p)\) for the set of all such maps.
\end{definition}

For \(D_i\in \Dcal_i^\alpha(p)\), the scalar \(\lambda_i(D_i)\) is uniquely determined. Indeed, taking \(h\equiv 1\) in Definition~\ref{def:frac-deriv-new} gives
\begin{equation}
\lambda_i(D_i)=D_i(I_i^\alpha 1).
\end{equation}
Thus the normalization parameter is part of the structure of the functional and not an additional choice. Because every element of \(\Acal_i^\alpha(U)\) has the form \(I_i^\alpha h\) for a unique \(h\in C(U)\), the identity in Definition~\ref{def:frac-deriv-new} determines \(D_i\) on its whole domain once \(\lambda_i(D_i)\) is fixed.

\begin{definition}[normalized \(i\)-th fractional derivation]
An element \(D_i\in \Dcal_i^\alpha(p)\) is said to be \emph{normalized} if \(\lambda_i(D_i)=1\).
\end{definition}

The term ``fractional derivation'' is used in this paper in a coordinate-anchored representational sense, not in the differential-graded-algebra sense. The spaces \(\Acal_i^\alpha(U)\) are not treated as algebras, no product structure is invoked on them, and no Leibniz rule is imposed. The objects of Definition~\ref{def:frac-deriv-new} should be thought of as coordinate fractional tangent generators selected by the anchored inversion property, and the use of the word ``derivation'' is justified only by analogy with the classical case where evaluation of a partial derivative at a point is the prototypical scalar derivation. 

\begin{theorem}[representation theorem]\label{thm:representation-new}
For each \(i\in\{1,\dots,d\}\), the space \(\Dcal_i^\alpha(p)\) is one-dimensional and
\begin{equation}
\Dcal_i^\alpha(p)=\R\cdot \bigl((\partial_{C,i}^\alpha)|_p\bigr).
\end{equation}
More precisely, every \(D_i\in \Dcal_i^\alpha(p)\) satisfies
\begin{equation}
D_i(f)=\lambda_i(D_i)\,(\partial_{C,i}^\alpha f)(p)
\qquad\text{for all } f\in \Acal_i^\alpha(U).
\end{equation}
\end{theorem}

\begin{proof}
Let \(D_i,E_i\in \Dcal_i^\alpha(p)\), and let \(a,b\in\R\). For every \(h\in C(U)\),
\begin{equation}
(aD_i+bE_i)(I_i^\alpha h)
=
a\lambda_i(D_i)h(p)+b\lambda_i(E_i)h(p)
=
\bigl(a\lambda_i(D_i)+b\lambda_i(E_i)\bigr)h(p).
\end{equation}
Hence \(aD_i+bE_i\in \Dcal_i^\alpha(p)\), so \(\Dcal_i^\alpha(p)\) is a vector space.

Now let \(f\in \Acal_i^\alpha(U)\). By Proposition~\ref{prop:reconstruction-new},
\begin{equation}
f=I_i^\alpha(\partial_{C,i}^\alpha f).
\end{equation}
Applying \(D_i\) gives
\begin{equation}
D_i(f)
=
D_i(I_i^\alpha(\partial_{C,i}^\alpha f))
=
\lambda_i(D_i)\,(\partial_{C,i}^\alpha f)(p).
\end{equation}
Thus every \(D_i\) is a scalar multiple of \((\partial_{C,i}^\alpha)|_p\). Finally, for \(h\in C(U)\),
\begin{equation}
(\partial_{C,i}^\alpha)|_p\,(I_i^\alpha h)
=
(\partial_{C,i}^\alpha(I_i^\alpha h))(p)
=
h(p),
\end{equation}
so \((\partial_{C,i}^\alpha)|_p\in \Dcal_i^\alpha(p)\). The claim follows.
\end{proof}

\begin{corollary}\label{cor:normalized-new}
For each \(i\), the unique normalized \(i\)-th fractional derivation is
\begin{equation}
(\partial_{C,i}^\alpha)|_p.
\end{equation}
\end{corollary}

\section{Fractional tangent space and dual basis}

\begin{definition}[fractional tangent space]
The fractional tangent space of order \(\alpha\) at \(p\) is
\begin{equation}
T_p^\alpha U := \bigoplus_{i=1}^d \Dcal_i^\alpha(p).
\end{equation}
For each \(i\in\{1,\dots,d\}\), let
\begin{equation}
e_i^\alpha(p)\in T_p^\alpha U
\end{equation}
denote the element whose \(i\)-th component is \((\partial_{C,i}^\alpha)|_p\) and whose other components are zero.
\end{definition}

The space \(T_p^\alpha U\) is a direct sum of the one-dimensional coordinate derivation spaces \(\Dcal_i^\alpha(p)\). Since the individual summands act a priori on different domains \(\Acal_i^\alpha(U)\), this construction should not be interpreted as a space of derivations on a single common function algebra in the classical differential-geometric sense. Its primary role is to package the coordinatewise anchored fractional directions into a finite-dimensional linear object indexed by the chosen coordinates.

The common anchored space \(\Acal^\alpha(U)=\bigcap_{i=1}^d \Acal_i^\alpha(U)\) of Definition~\ref{def:Ai} provides a natural common domain on which the elements of \(T_p^\alpha U\) act as scalar-valued linear functionals. At interior base points this action is faithful, as recorded in Proposition~\ref{prop:tangent-action-common} below.

\begin{lemma}[anchored power rule]\label{lem:anchored-power-rule}
Fix \(i\in\{1,\dots,d\}\), and let
\begin{equation}
M(x)=\prod_{k=1}^d (X_k^\alpha(x))^{m_k},
\qquad
m_k\in\mathbb Z_{\ge0},
\end{equation}
with \(m_i\ge1\). Then \(M\in\Acal_i^\alpha(U)\), and
\begin{equation}
\partial_{C,i}^\alpha M
=
\frac{\Gamma(m_i\alpha+1)}
{\Gamma((m_i-1)\alpha+1)\Gamma(1+\alpha)}
(X_i^\alpha)^{m_i-1}\prod_{j\neq i}(X_j^\alpha)^{m_j}.
\end{equation}
In particular, if \(m_k\ge1\) for every \(k\), then \(M\in\Acal^\alpha(U)\).
\end{lemma}

\begin{proof}
Only the \(i\)-th variable is relevant. Write
\begin{equation}
M=(X_i^\alpha)^{m_i}\prod_{j\neq i}(X_j^\alpha)^{m_j}.
\end{equation}
Set
\begin{equation}
H_i
:=
\frac{\Gamma(m_i\alpha+1)}
{\Gamma((m_i-1)\alpha+1)\Gamma(1+\alpha)}
(X_i^\alpha)^{m_i-1}\prod_{j\neq i}(X_j^\alpha)^{m_j}.
\end{equation}
The function \(H_i\) is continuous on \(U\), since \(m_i\ge1\) and \(\alpha>0\).
Using the standard Riemann--Liouville power formula
\begin{equation}
I_i^\alpha x_i^\beta
=
\frac{\Gamma(\beta+1)}{\Gamma(\beta+\alpha+1)}
x_i^{\beta+\alpha},
\qquad \beta>-1,
\end{equation}
with \(\beta=(m_i-1)\alpha\), and keeping the remaining variables fixed, we obtain
\begin{equation}
I_i^\alpha H_i
=
(X_i^\alpha)^{m_i}\prod_{j\neq i}(X_j^\alpha)^{m_j}
=
M.
\end{equation}
Hence \(M\in\Acal_i^\alpha(U)\). By the uniqueness of the anchored inverse in
Lemma~\ref{lem:Ai-inversion}, the unique continuous function whose
\(I_i^\alpha\)-integral is \(M\) is precisely \(H_i\). Therefore
\begin{equation}
\partial_{C,i}^\alpha M=H_i.
\end{equation}
If \(m_k\ge1\) for every \(k\), the same argument applies in every coordinate
direction, so \(M\in\bigcap_k\Acal_k^\alpha(U)=\Acal^\alpha(U)\).
\end{proof}

\begin{proposition}[action on the common anchored space]\label{prop:tangent-action-common}
Assume that the base point \(p\) is an interior point of \(U\), that is,
\begin{equation}
p_i>0\qquad\text{for all }i=1,\dots,d.
\end{equation}
Then each element
\begin{equation}
v=\sum_{i=1}^d c_i\, e_i^\alpha(p)\in T_p^\alpha U
\end{equation}
defines a linear functional
\begin{equation}
v:\Acal^\alpha(U)\to \R,
\qquad
v(f):=\sum_{i=1}^d c_i\,(\partial_{C,i}^\alpha f)(p).
\end{equation}
Moreover, the map
\begin{equation}
T_p^\alpha U\longrightarrow (\Acal^\alpha(U))^*,
\qquad
v\longmapsto v(\,\cdot\,),
\end{equation}
is injective. Equivalently, the functionals
\begin{equation}
f\longmapsto (\partial_{C,i}^\alpha f)(p),
\qquad i=1,\dots,d,
\end{equation}
are linearly independent on \(\Acal^\alpha(U)\).
\end{proposition}

\begin{proof}
Linearity is immediate from the linearity of the coordinatewise anchored Caputo derivatives.

It remains to prove injectivity. Set
\begin{equation}
q_i:=X_i^\alpha(p)=\frac{p_i^\alpha}{\Gamma(1+\alpha)}>0,
\qquad
P:=\prod_{k=1}^d q_k,
\end{equation}
and define, for \(j=1,\dots,d\),
\begin{equation}
f_j(x):=(X_j^\alpha(x))^2\prod_{k\neq j}X_k^\alpha(x).
\end{equation}
By Lemma~\ref{lem:anchored-power-rule}, each \(f_j\) belongs to \(\Acal^\alpha(U)\), because every coordinate exponent in \(f_j\) is at least one. The same lemma gives the required anchored derivatives of \(f_j\) in each coordinate direction.

We compute the matrix
\begin{equation}
M_{ij}:=(\partial_{C,i}^\alpha f_j)(p),
\qquad i,j=1,\dots,d.
\end{equation}
Using the one-dimensional power rule for the anchored Caputo derivative gives
\begin{equation}
M_{ij}
=
\begin{cases}
\displaystyle
\frac{\Gamma(2\alpha+1)}{\Gamma(1+\alpha)^2}\,
\prod_{k=1}^d q_k,
& i=j,\\[1.1em]
\displaystyle
q_j^2\prod_{k\neq i,j}q_k,
& i\neq j.
\end{cases}
\end{equation}
Equivalently, with
\begin{equation}
c_\alpha:=\frac{\Gamma(2\alpha+1)}{\Gamma(1+\alpha)^2},
\end{equation}
we may write
\begin{equation}
M_{ij}=P\,N_{ij},
\qquad
N_{ij}=
\begin{cases}
c_\alpha, & i=j,\\[0.3em]
q_j/q_i, & i\neq j.
\end{cases}
\end{equation}
The matrix \(N\) has the form
\begin{equation}
N=(c_\alpha-1)I+uv^T,
\qquad
u_i=\frac1{q_i},\quad v_j=q_j.
\end{equation}
We first note that \(c_\alpha>1\) for \(0<\alpha<1\). Indeed, the strict log-convexity of the Gamma function gives
\begin{equation}
\Gamma(1+\alpha)^2
<
\Gamma(1)\Gamma(1+2\alpha)
=
\Gamma(1+2\alpha),
\end{equation}
and hence \(c_\alpha>1\). Thus \(c_\alpha-1\neq0\). Since
\begin{equation}
v^Tu=\sum_{i=1}^d q_i\frac1{q_i}=d,
\end{equation}
the matrix determinant lemma, applied with the invertible matrix
\((c_\alpha-1)I\), yields
\begin{equation}
\det N
=
(c_\alpha-1)^d\left(1+\frac{v^Tu}{c_\alpha-1}\right)
=
(c_\alpha-1)^{d-1}(c_\alpha+d-1).
\end{equation}
Therefore
\begin{equation}
\det M
=
P^d\det N
=
P^d(c_\alpha-1)^{d-1}(c_\alpha+d-1)
\neq0.
\end{equation}

Now suppose
\begin{equation}
\sum_{i=1}^d c_i\,(\partial_{C,i}^\alpha f)(p)=0
\qquad\text{for every }f\in\Acal^\alpha(U).
\end{equation}
Testing this identity with \(f=f_j\), \(j=1,\dots,d\), gives
\begin{equation}
\sum_{i=1}^d c_i M_{ij}=0
\qquad j=1,\dots,d.
\end{equation}
Since \(M\) and hence \(M^T\) is invertible, \(c_1=\cdots=c_d=0\). This proves the asserted linear independence and therefore the injectivity of the map
\begin{equation}
T_p^\alpha U\to(\Acal^\alpha(U))^*.
\end{equation}
\end{proof}

The separating determinant computed above equals
\begin{equation}
\label{rem:faithful-all-alpha}
P^d(c_\alpha-1)^{d-1}(c_\alpha+d-1),
\qquad
c_\alpha=\frac{\Gamma(1+2\alpha)}{\Gamma(1+\alpha)^2}.
\end{equation}
By strict log-convexity of the Gamma function, \(c_\alpha>1\) for every
\(\alpha>0\). Hence, for every \(0<\alpha<1\) and every interior base point, the
action of \(T_p^\alpha U\) on \(\Acal^\alpha(U)\) is faithful. Moreover
\(c_\alpha\to2\) as \(\alpha\to1^-\), so the separating determinant has the
expected nonzero integer-order limit. We do not use, and do not claim here, a
stronger convergence theorem for the anchored spaces or operators as
\(\alpha\to1^-\).

\begin{remark}[boundary base points]\label{rem:boundary-base-points}
The interior assumption in Proposition~\ref{prop:tangent-action-common} is natural. If the base point lies on one or more anchoring hyperplanes, the common anchored space \(\Acal^\alpha(U)\) may no longer separate all coordinate fractional directions. Thus the abstract vector space \(T_p^\alpha U=\bigoplus_i\Dcal_i^\alpha(p)\) remains \(d\)-dimensional by definition, but its realization as functionals on the common anchored space can lose rank at boundary points.
\end{remark}

\begin{remark}\label{rem:Xi-not-in-common}
Note that the linear ``coordinate'' primitives \(X_i^\alpha\) themselves do \emph{not} belong to \(\Acal^\alpha(U)\) for \(d\ge2\), as observed in \eqref{rem:Xi-membership}. The common anchored space contains instead products and combinations that vanish on every coordinate hyperplane, such as \(\prod_{k=1}^d X_k^\alpha\). This is consistent with the fact that the natural test objects for the coordinate fractional derivations are anchored simultaneously in every variable.
\end{remark}

\begin{proposition}\label{prop:tangent-basis-new}
The family
\begin{equation}
\{e_1^\alpha(p),\dots,e_d^\alpha(p)\}
\end{equation}
is a basis of \(T_p^\alpha U\). In particular,
\begin{equation}
\dim T_p^\alpha U=d.
\end{equation}
\end{proposition}

\begin{proof}
By Theorem~\ref{thm:representation-new}, each \(\Dcal_i^\alpha(p)\) is one-dimensional and generated by \((\partial_{C,i}^\alpha)|_p\). Since \(T_p^\alpha U\) is the direct sum of these \(d\) one-dimensional spaces, the vectors \(e_i^\alpha(p)\) form a basis.
\end{proof}

Let \((T_p^\alpha U)^*\) denote the dual space.

\begin{definition}[fractional coordinate one-forms]\label{def:dual-basis-new}
For each \(j\in\{1,\dots,d\}\), define
\begin{equation}
d^\alpha X_j(p)\in (T_p^\alpha U)^*
\end{equation}
by the condition
\begin{equation}
d^\alpha X_j(p)\bigl(e_i^\alpha(p)\bigr)=\delta_{ij}
\qquad(i=1,\dots,d).
\end{equation}
\end{definition}

\begin{remark}
The notation \(d^\alpha X_j(p)\) refers here to the pointwise dual basis at \(p\). In the next section we use the same symbols \(d^\alpha X_j\) again as formal generators of a global exterior algebra. This notation is motivated by the pointwise dual basis, but the global exterior calculus is algebraic rather than bundle-theoretic.
\end{remark}

\section{Fractional differential forms}

At this stage it is essential to distinguish the local and global levels of the construction. In Sections~2--4, the objects \(T_p^\alpha U\) and \((T_p^\alpha U)^*\) are pointwise structures obtained from the coordinatewise anchored spaces \(\Acal_i^\alpha(U)\). In the present section, by contrast, we do \emph{not} construct a bundle of fractional differential forms over \(U\). Instead, we introduce a formal graded exterior algebra over a coefficient ring on which the relevant Caputo operators act transparently.

More precisely, the symbols \(d^\alpha X_1,\dots,d^\alpha X_d\) that appear below are formal degree-one generators modeled on the pointwise dual basis introduced in Definition~\ref{def:dual-basis-new}. The resulting algebraic exterior calculus is therefore guided by the local tangent--cotangent picture established earlier, but it is not claimed to be an intrinsic de~Rham-type theory on \(U\).

\subsection{The polynomial Caputo operator on fractional monomials}
\label{subsec:caputo-extension}

The local operator \(\partial_{C,i}^{\alpha}\) introduced in
Definition~\ref{def:anchored-caputo-new} is an inverse of \(I_i^\alpha\) on the
anchored range \(\Acal_i^\alpha(U)\). In the formal exterior complex below we
also need an operator on the polynomial algebra generated by the fractional
coordinate primitives. This operator agrees with the classical Caputo power
formula on monomials for which the Caputo integral is meaningful, but it is
defined algebraically on the whole polynomial coefficient ring. To avoid
confusing it with the anchored inverse, we denote it by \(\mathsf D_i^\alpha\).

For each \(i\in\{1,\dots,d\}\), recall
\begin{equation}
X_i^\alpha(x):=\frac{x_i^\alpha}{\Gamma(1+\alpha)}.
\end{equation}
Set
\begin{equation}
\mathscr P_\alpha(U):=\R[X_1^\alpha,\dots,X_d^\alpha].
\end{equation}

\begin{lemma}[algebraic independence of the fractional coordinates]
\label{lem:fractional-coordinates-independent}
The functions \(X_1^\alpha,\dots,X_d^\alpha\) are algebraically independent on \(U\). Equivalently, if a polynomial
\(P\in\R[y_1,\dots,y_d]\) satisfies
\begin{equation}
P(X_1^\alpha(x),\dots,X_d^\alpha(x))=0
\qquad\text{for every }x\in U,
\end{equation}
then \(P=0\).
\end{lemma}

\begin{proof}
The map
\begin{equation}
[0,b_1]\times\cdots\times[0,b_d]\ni x
\longmapsto
\left(\frac{x_1^\alpha}{\Gamma(1+\alpha)},\dots,
      \frac{x_d^\alpha}{\Gamma(1+\alpha)}\right)
\end{equation}
has image
\begin{equation}
\prod_{i=1}^d\left[0,\frac{b_i^\alpha}{\Gamma(1+\alpha)}\right],
\end{equation}
which contains a nonempty open rectangle in \(\R^d\). Hence, if
\(P(X_1^\alpha,\,\dots,X_d^\alpha)\) vanishes on \(U\), then \(P\) vanishes on a nonempty open subset of \(\R^d\). A real polynomial that vanishes on a nonempty open subset is the zero polynomial. Thus \(P=0\).
\end{proof}

{\begin{remark}[why the coefficient algebra is polynomial in \(X_i^\alpha\)]
\label{rem:why-polynomial-X-alpha}
The restriction to
\begin{equation}
\mathscr P_\alpha(U)=\R[X_1^\alpha,\dots,X_d^\alpha]
\end{equation}
is deliberate. On this coefficient algebra the polynomial Caputo operators lower
one exponent at a time, preserve the coefficient class, and commute in distinct
coordinate directions. These three properties are exactly what is needed for
\(d^\alpha\) to be well defined and nilpotent. For a general nonlinear function
of the variables \(X_i^\alpha\), the Caputo chain rule is not a finite
first-order rule of the classical type, and substantially more complicated
memory terms appear. Consequently, enlarging the coefficient class is not a
formal matter: one would have to specify a function space on which the relevant
fractional chain rules, mixed derivatives, and closure properties are available.
The present polynomial class should therefore be read as a controlled algebraic
model, not as a maximal domain for fractional forms.
\end{remark}}

\begin{definition}[polynomial Caputo operator]
\label{def:poly-caputo-operator}
By Lemma~\ref{lem:fractional-coordinates-independent}, each polynomial coefficient has a unique expression as a finite linear combination of monomials in \(X_1^\alpha,\dots,X_d^\alpha\). For a monomial
\begin{equation}
M(x)=\prod_{k=1}^d (X_k^\alpha(x))^{m_k},
\qquad m_k\in\mathbb Z_{\ge0},
\end{equation}
define
\begin{equation}
\mathsf D_i^\alpha M
:=
\begin{cases}
\dfrac{\Gamma(m_i\alpha+1)}
{\Gamma((m_i-1)\alpha+1)\Gamma(1+\alpha)}
(X_i^\alpha)^{m_i-1}\displaystyle\prod_{j\neq i}(X_j^\alpha)^{m_j},
& m_i\ge1,\\[1.1em]
0, & m_i=0.
\end{cases}
\end{equation}
Extend \(\mathsf D_i^\alpha\) to all of \(\mathscr P_\alpha(U)\) by
\(\R\)-linearity.
\end{definition}

For \(m_i\ge1\), the formula in Definition~\ref{def:poly-caputo-operator}
coincides with the absolutely convergent Caputo integral obtained by
differentiating \(x_i^{m_i\alpha}\) and using the Beta function identity. Thus
\(\mathsf D_i^\alpha\) is compatible with the usual Caputo power formula on
monomials. However, \(\mathsf D_i^\alpha\) is not, on its full domain, the same
object as the anchored inverse \(\partial_{C,i}^\alpha\), because
\(\mathscr P_\alpha(U)\) contains coefficients that need not lie in
\(\Acal_i^\alpha(U)\). On the intersection of their domains where the relevant
Caputo integral is admissible and the function is \(i\)-anchored, the two
operators agree by Proposition~\ref{prop:anchored-vs-classical-new}.

\begin{lemma}[stability and mixed commutation on \(\mathscr P_\alpha\)]
\label{lem:poly-stability-commutation}
For every \(i\),
\begin{equation}
\mathsf D_i^\alpha(\mathscr P_\alpha(U))\subset \mathscr P_\alpha(U).
\end{equation}
Moreover, for all \(i,j\),
\begin{equation}
\mathsf D_i^\alpha\mathsf D_j^\alpha
=
\mathsf D_j^\alpha\mathsf D_i^\alpha
\qquad\text{on }\mathscr P_\alpha(U).
\end{equation}
\end{lemma}

\begin{proof}
Stability follows directly from Definition~\ref{def:poly-caputo-operator}.
For commutation, it suffices by linearity to check monomials
\(M=\prod_k(X_k^\alpha)^{m_k}\). If \(i=j\), the assertion is tautological. If
\(i\neq j\), applying \(\mathsf D_i^\alpha\) lowers only the \(i\)-th exponent
and multiplies by a coefficient depending only on \(m_i\); applying
\(\mathsf D_j^\alpha\) lowers only the \(j\)-th exponent and multiplies by a
coefficient depending only on \(m_j\). Hence the two possible orders produce the
same monomial and the same product of scalar coefficients. The cases
\(m_i=0\) or \(m_j=0\) are included, since then both corresponding compositions
vanish. Thus the operators commute on monomials and therefore on all of
\(\mathscr P_\alpha(U)\).
\end{proof}

\subsection{Formal differential forms}

In this subsection the symbols \(d^\alpha X_j\) are used as formal degree-one generators. They are modeled on, but distinct from, the pointwise dual covectors \(d^\alpha X_j(p)\) of Definition~\ref{def:dual-basis-new}.

Let \(d^\alpha X_1,\dots,d^\alpha X_d\) be formal degree-one generators satisfying the usual skew-commutation relations
\begin{equation}
d^\alpha X_i\wedge d^\alpha X_j = - d^\alpha X_j\wedge d^\alpha X_i,
\qquad
d^\alpha X_i\wedge d^\alpha X_i=0.
\end{equation}
For a multi-index \(I=(i_1<\cdots<i_p)\), write
\begin{equation}
d^\alpha X_I:=d^\alpha X_{i_1}\wedge\cdots\wedge d^\alpha X_{i_p}.
\end{equation}

For each \(p\ge0\), define
\begin{equation}
\Omega_\alpha^p(U;\mathscr P_\alpha)
:=
\bigoplus_{|I|=p}\mathscr P_\alpha(U)\,d^\alpha X_I.
\end{equation}
The graded algebra
\begin{equation}
\Omega_\alpha^\bullet(U;\mathscr P_\alpha)
:=
\bigoplus_{p=0}^d \Omega_\alpha^p(U;\mathscr P_\alpha)
\end{equation}
is the formal fractional exterior algebra used below.

\begin{definition}[fractional exterior differential]\label{def:dalpha-new}
Let
\begin{equation}
\omega=\sum_{|I|=p} a_I\,d^\alpha X_I
\in \Omega_\alpha^p(U;\mathscr P_\alpha).
\end{equation}
Define
\begin{equation}
d^\alpha\omega
:=
\sum_{|I|=p}\sum_{i=1}^d
(\mathsf D_i^\alpha a_I)\,d^\alpha X_i\wedge d^\alpha X_I.
\end{equation}
\end{definition}

By Lemma~\ref{lem:poly-stability-commutation}, the operator
\begin{equation}
d^\alpha:\Omega_\alpha^p(U;\mathscr P_\alpha)\to \Omega_\alpha^{p+1}(U;\mathscr P_\alpha)
\end{equation}
is well defined.

{\begin{definition}[formal fractional interior product]
\label{def:fractional-interior-product}
Let
\begin{equation}
V=\sum_{i=1}^d V^i\,\mathsf e_i^\alpha,
\qquad V^i\in\mathscr P_\alpha(U),
\end{equation}
be a formal fractional vector field, where
\(\mathsf e_i^\alpha\) denotes the formal vector dual to \(d^\alpha X_i\).
The formal interior product
\begin{equation}
\iota_V^\alpha:
\Omega_\alpha^p(U;\mathscr P_\alpha)
\longrightarrow
\Omega_\alpha^{p-1}(U;\mathscr P_\alpha)
\end{equation}
is defined on monomial forms by
\begin{equation}
\iota_V^\alpha
\left(
a\,d^\alpha X_{i_1}\wedge\cdots\wedge d^\alpha X_{i_p}
\right)
:=
\sum_{r=1}^p
(-1)^{r-1}
a\,V^{i_r}\,
d^\alpha X_{i_1}\wedge\cdots
\widehat{d^\alpha X_{i_r}}
\cdots\wedge d^\alpha X_{i_p},
\end{equation}
and extended \(\mathscr P_\alpha(U)\)-linearly.
\end{definition}

\begin{remark}
The operator \(\iota_V^\alpha\) is the algebraic contraction determined by the
chosen fractional coframe. It is fractional only in the sense that it is attached
to the formal dual pair
\((\mathsf e_i^\alpha,d^\alpha X_i)\). It is not a new nonlocal differential
operator and does not involve the Caputo rule. This is consistent with the
coframe-attached nature of the present exterior calculus.
\end{remark}

\begin{remark}[fractional integration of formal top forms]
\label{rem:fractional-integration-nonadditivity}
One can attach an iterated left Riemann--Liouville integral to a formal top form
on the anchored rectangle. For example, for
\begin{equation}
\omega=a\,d^\alpha X_1\wedge\cdots\wedge d^\alpha X_d,
\qquad a\in\mathscr P_\alpha(U),
\end{equation}
a natural anchored integral is
\begin{equation}
\int_U^\alpha \omega
:=
(I_1^\alpha\cdots I_d^\alpha a)(b_1,\dots,b_d).
\end{equation}
This convention is compatible with the anchored choice of lower terminal
\(0\) in each coordinate. However, it should not be confused with an additive
measure-theoretic integration of differential forms. Already in one dimension,
for \(0<c<b\) and \(f\equiv1\),
\begin{equation}
\frac{1}{\Gamma(\alpha)}
\int_0^b (b-\xi)^{\alpha-1}\,d\xi
=
\frac{b^\alpha}{\Gamma(1+\alpha)}
\end{equation}
is not equal in general to
\begin{equation}
\frac{c^\alpha}{\Gamma(1+\alpha)}
+
\frac{(b-c)^\alpha}{\Gamma(1+\alpha)}.
\end{equation}
Thus fractional integration with a fixed memory terminal is not additive under
subdivision of the interval. For this reason the present paper uses fractional
integration only as an anchored analytic background for the operators
\(I_i^\alpha\), and does not claim a full Stokes theory for integrated
fractional forms.
\end{remark}}

{The ordinary wedge product on
\(\Omega_\alpha^\bullet(U;\mathscr P_\alpha)\) is the exterior product generated
by the skew-commuting symbols \(d^\alpha X_i\), together with the ordinary
multiplication of polynomial coefficients. With this product the space of forms
is a graded algebra, but \(d^\alpha\) is not a graded derivation. Thus the
present construction is algebraically complete as a nilpotent chain complex, but
not as an ordinary differential graded algebra. This is consistent with the
generalized Leibniz formulae for fractional derivatives, which involve
non-finite expansions rather than the two-term product rule; see, for example,
\cite[Section~15]{SamkoKilbasMarichev1993}. If desired, one may instead
transport the ordinary wedge product through the chain isomorphism \(\Phi\) of
Proposition~\ref{prop:dRham-reduction} by setting
\begin{equation}
\omega\wedge_\Phi\eta
:=
\Phi^{-1}\bigl(\Phi(\omega)\wedge\Phi(\eta)\bigr).
\end{equation}
With respect to this transported product, \(d^\alpha\) becomes a graded
derivation by construction. However, the product \(\wedge_\Phi\) is no longer
the ordinary coefficientwise wedge product, and we do not use it in the Maxwell
section.}

\begin{proposition}[failure of the ordinary graded Leibniz rule]
\label{prop:not-ordinary-dga}
For \(0<\alpha<1\), the operator \(d^\alpha\) is not a graded derivation of the
ordinary wedge product on
\(\Omega_\alpha^\bullet(U;\mathscr P_\alpha)\). In particular,
\begin{equation}
d^\alpha(\omega\wedge\eta)
\neq
d^\alpha\omega\wedge\eta
+
(-1)^{\deg\omega}\omega\wedge d^\alpha\eta
\end{equation}
in general.
\end{proposition}

\begin{proof}
It is enough to consider \(0\)-forms. Let \(a=b=X_i^\alpha\). Then
\begin{equation}
ab=(X_i^\alpha)^2.
\end{equation}
By Definition~\ref{def:poly-caputo-operator},
\begin{equation}
d^\alpha\bigl((X_i^\alpha)^2\bigr)
=
\frac{\Gamma(2\alpha+1)}{\Gamma(1+\alpha)^2}
X_i^\alpha\,d^\alpha X_i.
\end{equation}
If \(d^\alpha\) satisfied the ordinary Leibniz rule, we would instead have
\begin{equation}
d^\alpha(ab)
=
a\,d^\alpha b+b\,d^\alpha a
=
2X_i^\alpha\,d^\alpha X_i.
\end{equation}
The two coefficients agree only in the integer-order limit \(\alpha=1\). Indeed,
for
\begin{equation}
c_\alpha:=\frac{\Gamma(2\alpha+1)}{\Gamma(1+\alpha)^2}
\end{equation}
one has \(c_1=2\), and \(c_\alpha<2\) for \(0<\alpha<1\). To see this, note that
\begin{equation}
\frac{d}{d\alpha}\log c_\alpha
=
2\psi(2\alpha+1)-2\psi(\alpha+1)>0,
\end{equation}
because the digamma function \(\psi\) is strictly increasing on \((0,\infty)\).
Thus \(c_\alpha\) is strictly increasing on \((0,1]\), and so \(c_\alpha\neq2\)
for \(0<\alpha<1\). Hence the ordinary graded Leibniz rule fails.
\end{proof}

\begin{example}[A worked computation in dimension \(2\)]\label{ex:worked}
Let \(U=[0,b_1]\times[0,b_2]\subset\R^2\), and consider the \(1\)-form
\begin{equation}
\omega
=
X_1^\alpha X_2^\alpha\, d^\alpha X_1
+
(X_1^\alpha)^2\, d^\alpha X_2
\in \Omega_\alpha^1(U;\mathscr P_\alpha).
\end{equation}
We compute \(d^\alpha\omega\) explicitly.

Write
\begin{equation}
a_1:=X_1^\alpha X_2^\alpha,
\qquad
a_2:=(X_1^\alpha)^2.
\end{equation}
Then, by Definition~\ref{def:dalpha-new},
\begin{equation}
d^\alpha\omega
=
\sum_{i=1}^2 (\mathsf D_i^\alpha a_1)\, d^\alpha X_i\wedge d^\alpha X_1
+
\sum_{i=1}^2 (\mathsf D_i^\alpha a_2)\, d^\alpha X_i\wedge d^\alpha X_2.
\end{equation}
Because \(d^\alpha X_1\wedge d^\alpha X_1=0\) and
\(d^\alpha X_2\wedge d^\alpha X_2=0\), this reduces to
\begin{equation}
d^\alpha\omega
=
(\mathsf D_1^\alpha a_2)\, d^\alpha X_1\wedge d^\alpha X_2
+
(\mathsf D_2^\alpha a_1)\, d^\alpha X_2\wedge d^\alpha X_1.
\end{equation}
Using \(d^\alpha X_2\wedge d^\alpha X_1=-\,d^\alpha X_1\wedge d^\alpha X_2\),
we obtain
\begin{equation}
d^\alpha\omega
=
\bigl(\mathsf D_1^\alpha a_2-\mathsf D_2^\alpha a_1\bigr)\,
d^\alpha X_1\wedge d^\alpha X_2.
\end{equation}

By the monomial rule in Definition~\ref{def:poly-caputo-operator}, the
\(x_2\)-exponent of \(X_1^\alpha X_2^\alpha\) is one. Hence
\begin{equation}
\mathsf D_2^\alpha a_1
=
\mathsf D_2^\alpha(X_1^\alpha X_2^\alpha)
=
X_1^\alpha.
\end{equation}
Similarly,
\begin{equation}
\mathsf D_1^\alpha a_2
=
\mathsf D_1^\alpha (X_1^\alpha)^2
=
\frac{\Gamma(2\alpha+1)}{\Gamma(1+\alpha)^2}\,X_1^\alpha.
\end{equation}
Therefore
\begin{equation}
d^\alpha\omega
=
\left(
\frac{\Gamma(2\alpha+1)}{\Gamma(1+\alpha)^2}-1
\right)
X_1^\alpha\,
d^\alpha X_1\wedge d^\alpha X_2.
\end{equation}

This example illustrates two basic features of the formal fractional exterior calculus: first, the operator \(d^\alpha\) differentiates only the coefficients by the polynomial rule \(\mathsf D_i^\alpha\); second, the antisymmetry of the wedge product reorganizes the result exactly as in the classical coordinate formula for the exterior derivative.
\end{example}

\section{Nilpotency of the fractional exterior differential}

In a general fractional setting, the identity \((d^\alpha)^2=0\) cannot be taken for granted, because mixed Caputo derivatives need not commute on arbitrary coefficient classes. The role of the polynomial algebra \(\mathscr P_\alpha(U)\) is precisely to provide a concrete coefficient class on which this commutation is automatic.

\begin{theorem}\label{thm:nilpotency-new}
For every
\begin{equation}
\omega\in \Omega_\alpha^p(U;\mathscr P_\alpha),
\end{equation}
one has
\begin{equation}
(d^\alpha)^2\omega=0.
\end{equation}
\end{theorem}

\begin{proof}
Write
\begin{equation}
\omega=\sum_{|I|=p}a_I\,d^\alpha X_I,
\qquad a_I\in\mathscr P_\alpha(U).
\end{equation}
Using Definition~\ref{def:dalpha-new} twice, we obtain
\begin{equation}
(d^\alpha)^2\omega
=
\sum_{|I|=p}\sum_{i,j=1}^d
(\mathsf D_j^\alpha\mathsf D_i^\alpha a_I)\,
d^\alpha X_j\wedge d^\alpha X_i\wedge d^\alpha X_I.
\end{equation}
The terms with \(i=j\) vanish because
\(d^\alpha X_i\wedge d^\alpha X_i=0\). For \(i\neq j\),
Lemma~\ref{lem:poly-stability-commutation} gives
\begin{equation}
\mathsf D_j^\alpha\mathsf D_i^\alpha a_I
=
\mathsf D_i^\alpha\mathsf D_j^\alpha a_I.
\end{equation}
The \((i,j)\)-term and the \((j,i)\)-term therefore have the same coefficient
and opposite exterior basis factor, since
\begin{equation}
d^\alpha X_j\wedge d^\alpha X_i
=
-\,d^\alpha X_i\wedge d^\alpha X_j.
\end{equation}
They cancel pairwise. Hence \((d^\alpha)^2\omega=0\).
\end{proof}

The following proposition delimits the role of the formal exterior
calculus precisely. The operator \(d^\alpha\) is genuinely distinct from the
ordinary exterior derivative on the algebraic polynomial complex \(\Omega_{\mathrm{poly}}^\bullet(\R^d_y)\) --- its monomial
coefficients are the Gamma ratios \(g(m)\) rather than the integers \(m\) ---
but the two are conjugate through the positive diagonal rescaling \(\Phi\), so
\(d^\alpha\) carries no cohomological information beyond the classical case. In
particular the Maxwell-type continuity law of Section~7, which follows from
\((d^\alpha)^2=0\) alone, does not detect the fractional order. The role of the
polynomial algebra \(\mathscr P_\alpha(U)\) is therefore to furnish a concrete,
transparently nilpotent coefficient class on which the local anchored structure
of Sections~2--4 is realized globally, not to produce new cohomology. The
genuinely fractional content of the paper resides in the local theory, see again \eqref{rem:faithful-all-alpha}.

\begin{proposition}[Algebraic reduction to the polynomial de~Rham complex]
\label{prop:dRham-reduction}
Let
\begin{equation}
\Omega_{\mathrm{poly}}^\bullet(\R^d_y)
:=
\R[y_1,\dots,y_d]\otimes
\Lambda^\bullet(dy_1,\dots,dy_d)
\end{equation}
be the ordinary polynomial de~Rham complex in the algebraically independent
variables \(y_1,\dots,y_d\), with differential \(d\). Define the graded
\(\R\)-linear map
\begin{equation}
\Phi:\Omega_\alpha^\bullet(U;\mathscr P_\alpha)
\longrightarrow
\Omega_{\mathrm{poly}}^\bullet(\R^d_y)
\end{equation}
on monomial basis elements by
\begin{equation}
\Phi\Bigl(
\prod_{k=1}^d (X_k^\alpha)^{m_k}\,d^\alpha X_I
\Bigr)
:=
\Bigl(\prod_{k=1}^d\phi(m_k)\Bigr)
y^{\mathbf m}\,dy_I,
\end{equation}
where
\begin{equation}
\phi(m):=
\frac{1}{m!}\,
\frac{\Gamma(m\alpha+1)}{\Gamma(1+\alpha)^m},
\qquad
\phi(0):=1.
\end{equation}
Then each \(\phi(m)>0\), so \(\Phi\) is a graded vector-space isomorphism, and
\begin{equation}
\Phi\circ d^\alpha=d\circ\Phi.
\end{equation}
\end{proposition}

\begin{proof}
By Lemma~\ref{lem:fractional-coordinates-independent}, the displayed fractional monomials form a basis of the coefficient algebra \(\mathscr P_\alpha(U)\).

Since \(\Phi\) is diagonal on this basis with strictly positive diagonal entries,
it is a graded vector-space isomorphism. It remains to verify the chain relation.
It suffices to check a monomial \(0\)-form
\begin{equation}
M=\prod_k(X_k^\alpha)^{m_k}.
\end{equation}
By Definition~\ref{def:poly-caputo-operator},
\begin{equation}
d^\alpha M
=
\sum_{i:m_i\ge1}
g(m_i)
(X_i^\alpha)^{m_i-1}
\prod_{j\neq i}(X_j^\alpha)^{m_j}
\,d^\alpha X_i,
\end{equation}
where
\begin{equation}
g(m):=
\frac{\Gamma(m\alpha+1)}
{\Gamma((m-1)\alpha+1)\Gamma(1+\alpha)}.
\end{equation}
Hence
\begin{equation}
\Phi(d^\alpha M)
=
\sum_{i:m_i\ge1}
g(m_i)\phi(m_i-1)\prod_{j\neq i}\phi(m_j)\,
y^{\mathbf m-e_i}\,dy_i.
\end{equation}
On the other hand,
\begin{equation}
d(\Phi M)
=
\sum_{i:m_i\ge1}
m_i\prod_k\phi(m_k)\,
y^{\mathbf m-e_i}\,dy_i.
\end{equation}
The coefficients agree because
\begin{equation}
g(m)\phi(m-1)=m\phi(m),
\end{equation}
which follows immediately from the definition of \(\phi\). For a general
monomial form \(M\,d^\alpha X_I\), both differentials act only on the coefficient
\(M\) and wedge the resulting \(1\)-form on the left with the fixed exterior
basis element. Therefore the same computation proves
\(\Phi d^\alpha=d\Phi\) in every degree.
\end{proof}

\subsection{A fractional polynomial Poincar\'e homotopy}

The chain isomorphism of Proposition~\ref{prop:dRham-reduction} gives an
immediate cohomological consequence. For completeness, and also to make the
de~Rham nature of the formal complex explicit, we record the corresponding
fractional polynomial Poincar\'e homotopy. Homotopy operators for exterior differential systems are a classical tool, but
they also remain useful in recent geometric and mathematical-physics contexts;
for example, Kycia and \v{S}ilhan use the linear homotopy operator of the
Poincar\'e lemma to invert covariant exterior derivatives locally. The operator
$K_\alpha$ below is the corresponding diagonal-rescaled homotopy in the
anchored fractional polynomial complex \cite{KyciaSilhan2025}.

For \(m\in\mathbb Z_{\ge0}\), recall
\begin{equation}
\phi(m)=
\frac{1}{m!}\,
\frac{\Gamma(m\alpha+1)}{\Gamma(1+\alpha)^m},
\qquad
\phi(0)=1,
\end{equation}
and for a multi-index \(\mathbf m=(m_1,\dots,m_d)\) write
\begin{equation}
\phi(\mathbf m):=\prod_{j=1}^d\phi(m_j),
\qquad
(X^\alpha)^{\mathbf m}:=\prod_{j=1}^d (X_j^\alpha)^{m_j}.
\end{equation}

\begin{definition}[fractional polynomial homotopy operator]
\label{def:fractional-homotopy}
Let
\begin{equation}
\omega=(X^\alpha)^{\mathbf m}\,
d^\alpha X_{i_1}\wedge\cdots\wedge d^\alpha X_{i_p},
\qquad
 i_1<\cdots<i_p,
\end{equation}
be a monomial \(p\)-form. Define \(K_\alpha\omega=0\) if \(p=0\). If \(p\ge1\),
define
\begin{equation}
K_\alpha\omega
:=
\frac{1}{|\mathbf m|+p}
\sum_{r=1}^p
(-1)^{r-1}
\frac{\phi(\mathbf m)}
{\phi(\mathbf m+e_{i_r})}
(X^\alpha)^{\mathbf m+e_{i_r}}\,
d^\alpha X_{i_1}\wedge\cdots
\widehat{d^\alpha X_{i_r}}
\cdots\wedge d^\alpha X_{i_p},
\end{equation}
where \(e_j\) is the \(j\)-th standard multi-index and
\(|\mathbf m|=m_1+\cdots+m_d\). We extend \(K_\alpha\) linearly to all polynomial
fractional forms.
\end{definition}

The following theorem gives an explicit contracting homotopy,
not merely an abstract isomorphism with the classical polynomial de~Rham complex.
This strengthens the interpretation of the formal fractional exterior algebra as
a genuine chain-level object. At the same time, the result confirms that the
polynomial complex has classical cohomology; the fractional information is
encoded in the anchored local operators, the polynomial Caputo operators, the Gamma-weighted differential, and the
Gamma-weighted homotopy, not in new polynomial cohomology classes.

\begin{theorem}[fractional polynomial Poincar\'e lemma]
\label{thm:fractional-poincare}
Let
\begin{equation}
\pi_\alpha:
\Omega_\alpha^\bullet(U;\mathscr P_\alpha)
\longrightarrow
\Omega_\alpha^\bullet(U;\mathscr P_\alpha)
\end{equation}
be the projection onto constant \(0\)-forms, and zero on all positive form
degrees. Then
\begin{equation}
d^\alpha K_\alpha+K_\alpha d^\alpha
=
\mathrm{id}-\pi_\alpha.
\end{equation}
Consequently,
\begin{equation}
H^0\bigl(\Omega_\alpha^\bullet(U;\mathscr P_\alpha),d^\alpha\bigr)
\simeq \mathbb R,
\qquad
H^p\bigl(\Omega_\alpha^\bullet(U;\mathscr P_\alpha),d^\alpha\bigr)=0
\quad(p\ge1).
\end{equation}
\end{theorem}

\begin{proof}
Let
\begin{equation}
\Phi:
\Omega_\alpha^\bullet(U;\mathscr P_\alpha)
\to
\Omega_{\mathrm{poly}}^\bullet(\R^d_y)
\end{equation}
be the chain isomorphism from Proposition~\ref{prop:dRham-reduction}. On the
ordinary polynomial de~Rham complex in the variables \(y_i\), let
\begin{equation}
E=\sum_{i=1}^d y_i\frac{\partial}{\partial y_i}
\end{equation}
be the Euler vector field. If \(\eta\) is homogeneous of polynomial degree
\(r\) and form degree \(p\), define
\begin{equation}
K\eta
:=
\begin{cases}
\dfrac{1}{r+p}\,\iota_E\eta, & r+p>0,\\[0.8em]
0, & r=p=0.
\end{cases}
\end{equation}
By Cartan's formula,
\begin{equation}
d\iota_E+\iota_Ed=L_E.
\end{equation}
On a homogeneous polynomial \(p\)-form of polynomial degree \(r\), the Lie
derivative \(L_E\) acts by multiplication by \(r+p\). Therefore
\begin{equation}
dK+Kd=\mathrm{id}-\pi,
\end{equation}
where \(\pi\) is projection onto constant \(0\)-forms.

By construction,
\begin{equation}
K_\alpha=\Phi^{-1}K\Phi.
\end{equation}
Indeed, applying \(\Phi\), then \(K\), and then \(\Phi^{-1}\) to a monomial
fractional form gives exactly the coefficient in
Definition~\ref{def:fractional-homotopy}. Since
\begin{equation}
\Phi d^\alpha=d\Phi,
\end{equation}
we obtain
\begin{equation}
d^\alpha K_\alpha+K_\alpha d^\alpha
=
\Phi^{-1}(dK+Kd)\Phi
=
\Phi^{-1}(\mathrm{id}-\pi)\Phi
=
\mathrm{id}-\pi_\alpha.
\end{equation}
The cohomology statement follows immediately.
\end{proof}

\subsection{Rigidity of admissible linear coordinate changes}

The preceding reduction to the polynomial de~Rham complex might suggest that the
anchored formal calculus is merely a disguised classical calculus in the
variables \(y_i=X_i^\alpha\). The following result shows that this interpretation
has to be treated with care. In the real-valued anchored setting, linear changes
that preserve the fractional coordinate span are rigid once the relevant
fractional powers are required to remain real-valued on the anchored box.

In particular, this explains why the present
calculus is genuinely anchored and coordinate-sensitive. In the classical case
\(\alpha=1\), arbitrary invertible linear changes of coordinates preserve the
linear span of the coordinate functions. For \(0<\alpha<1\), mixing two
coordinates before taking the fractional power produces mixed dependence that
cannot be represented as a linear combination of the separate fractional
primitives \(X_j^\alpha\). In the real-valued anchored setting considered here,
the admissible linear maps are therefore restricted to positive-cone-preserving
maps, and preservation of the fractional coordinate span forces coordinate
permutations and positive rescalings.

\begin{proposition}[rigidity of real fractional linear coordinates]
\label{prop:fractional-coordinate-rigidity}
Let \(0<\alpha<1\), and let \(A=(a_{ij})\in\R^{d\times d}\). Assume that each row
functional
\begin{equation}
L_r(x):=(Ax)_r=\sum_{j=1}^d a_{rj}x_j
\end{equation}
is nonnegative on \(U\), so that \(L_r(x)^\alpha\) is real-valued on \(U\). For
\(r=1,\dots,d\), set
\begin{equation}
Y_r^\alpha(x)
:=
\frac{L_r(x)^\alpha}{\Gamma(1+\alpha)}.
\end{equation}
Assume that, for every \(r\), the function \(Y_r^\alpha\) lies in the linear span
of
\begin{equation}
X_1^\alpha,\dots,X_d^\alpha,
\qquad
X_j^\alpha(x)=\frac{x_j^\alpha}{\Gamma(1+\alpha)}.
\end{equation}
Then each row of \(A\) contains at most one nonzero entry. In particular, if
\(A\) is invertible, then \(A=DP\), where \(D\) is a positive diagonal matrix and
\(P\) is a permutation matrix.

Conversely, every such matrix \(DP\) preserves the linear span of
\(X_1^\alpha,\dots,X_d^\alpha\).
\end{proposition}

\begin{proof}
First observe that the nonnegativity of \(L_r\) on \(U\) forces
\(a_{rj}\ge0\) for every \(j\). Indeed, setting all coordinates except
\(x_j=s\) equal to zero gives
\begin{equation}
L_r(se_j)=a_{rj}s\ge0
\end{equation}
for all \(s\in[0,b_j]\), hence \(a_{rj}\ge0\).

Fix a row \(r\). By assumption, there exist constants \(c_1,\dots,c_d\) such
that
\begin{equation}
L_r(x)^\alpha=\sum_{j=1}^d c_jx_j^\alpha
\end{equation}
on \(U\). Suppose, toward a contradiction, that two entries in the row are
positive, say \(a_{rm}>0\) and \(a_{rn}>0\) with \(m\neq n\). Restrict the
identity to the coordinate two-plane
\begin{equation}
x_m=s,
\qquad
x_n=t,
\qquad
x_k=0\quad(k\neq m,n).
\end{equation}
Then, for \(s,t>0\),
\begin{equation}
(a_{rm}s+a_{rn}t)^\alpha
=
c_ms^\alpha+c_nt^\alpha.
\end{equation}
Both sides are \(C^2\) on \((0,b_m]\times(0,b_n]\). Taking the mixed derivative
with respect to \(s\) and \(t\) gives
\begin{equation}
\alpha(\alpha-1)a_{rm}a_{rn}
(a_{rm}s+a_{rn}t)^{\alpha-2}
=
0,
\end{equation}
which is impossible because \(0<\alpha<1\) and \(a_{rm}a_{rn}>0\). Hence each row
has at most one nonzero entry.

If \(A\) is invertible, no row and no column can be zero. Since each row has at
most one nonzero entry, invertibility implies that each row and each column has
exactly one nonzero entry. Therefore \(A=DP\) for a positive diagonal matrix
\(D\) and a permutation matrix \(P\). Conversely, if \(A=DP\), then each
\(Y_r^\alpha\) is a positive scalar multiple of one of the \(X_j^\alpha\), so the
span is preserved.
\end{proof}

{\begin{remark}[coordinate changes and loss of power-law kernels]
\label{rem:coordinate-changes-pseudodifferential}
Proposition~\ref{prop:fractional-coordinate-rigidity} should also be understood
from the operator point of view. The left fractional integral
\begin{equation}
(I_i^\alpha f)(x)
=
\frac{1}{\Gamma(\alpha)}
\int_0^{x_i}(x_i-\xi)^{\alpha-1}
f(x_1,\dots,\xi,\dots,x_d)\,d\xi
\end{equation}
has a power-law memory kernel tied to the coordinate line and to the anchoring
hyperplane \(\{x_i=0\}\). Under a general change of variables this kernel is not
mapped to another one-dimensional power-law kernel in a new coordinate
direction. Instead, the transformed operator becomes a more complicated
nonlocal, generally pseudo-differential or integral operator depending on the
full transformation. Thus the present calculus is not invariant under arbitrary
coordinate transformations. The admissible transformations are precisely those
that preserve the anchored power-law structure of the chosen coframe, namely
coordinate permutations and positive rescalings in the real-valued linear
setting of Proposition~\ref{prop:fractional-coordinate-rigidity}.
\end{remark}}

\begin{remark}
The polynomial coefficient algebra is not claimed to be maximal. It is a concrete class on which the fractional exterior differential is closed and nilpotent with a transparent proof. The restriction is not merely cosmetic: on substantially larger coefficient classes, such as arbitrary anchored functions with only minimal regularity, mixed Caputo derivatives need not be available or need not commute without additional assumptions. Identifying broader natural coefficient algebras on which \(d^\alpha\) remains closed and nilpotent is therefore a genuine open direction.
\end{remark}

\begin{remark}[Relation with the classical coordinate expressions]
\label{rem:alpha-to-one}
We briefly record how the present construction relates to the ordinary coordinate formulas as \(\alpha\to1^-\), on the restricted classes considered here.

First, for fixed \(x\in U\) with \(x_i>0\),
\begin{equation}
X_i^\alpha(x)=\frac{x_i^\alpha}{\Gamma(1+\alpha)}\longrightarrow x_i
\qquad\text{as }\alpha\to1^-,
\end{equation}
since \(\Gamma(2)=1\). Second, for every fixed monomial family
\(M_\alpha=\prod_i (X_i^\alpha)^{m_i}\), the coefficient in the polynomial rule
\(\mathsf D_i^\alpha M_\alpha\) converges to the corresponding ordinary
coordinate coefficient as \(\alpha\to1^-\). This is a statement about the
coordinate formulas on the polynomial families considered here, not a convergence
theorem for the full anchored spaces or for a full de~Rham complex.

Accordingly, on the polynomial coefficient algebra used in this paper, the defining coordinate expressions converge to their ordinary first-order counterparts away from the coordinate hyperplanes. This should be interpreted as recovery of the corresponding \emph{coordinate formulas} on the restricted classes considered here, rather than as a proof that the present construction reproduces the full classical tangent bundle or the full de~Rham complex.
\end{remark}

\section{A Lorentzian coframe-attached Maxwell-type system on a spacetime box}

Fractional differential forms have already been used to formulate Maxwell-type systems. In particular, Baleanu, Golmankhaneh, Golmankhaneh, and Baleanu introduce a fractional exterior derivative containing both spatial and temporal Caputo derivatives and derive fractional Maxwell equations, charge conservation, and wave equations \cite{BaleanuMaxwell}. This development is part of a broader literature on fractional electromagnetics involving fractional multipoles, fractional curl operators, and fractional vector-calculus formulations; see, for example, \cite{Engheta1996,Engheta1997,Engheta1998,Tarasov2008}. The present section does not attempt to reproduce the full physical theory. Instead, we show that the formal exterior calculus developed above supports a mathematically consistent Lorentzian coframe-attached fractional Maxwell-type system on a bounded spacetime box within the polynomial coefficient class considered here. The emphasis is on the differential-form structure, on the continuity law that follows from \((d^\alpha)^2=0\), and on the associated fractional wave equations.

{The resulting equations should therefore not be identified with fractional electromagnetic models in which fractional vector-calculus operators act directly on general physical fields. Here the operators are the polynomial Caputo operators attached to the fixed fractional coframe \((d^\alpha\Theta,d^\alpha X_1,d^\alpha X_2,d^\alpha X_3)\), and the fields are restricted to \(\mathscr P_\alpha(U_{\mathrm{st}})\).}

\subsection{Spacetime notation}

Let
\begin{equation}
U_{\mathrm{st}}:=[0,T]\times[0,b_1]\times[0,b_2]\times[0,b_3]\subset\R^4
\end{equation}
be a spacetime box, with coordinates
\begin{equation}
(t,x_1,x_2,x_3).
\end{equation}
Define the fractional coordinate primitives
\begin{equation}
\Theta^\alpha(t):=\frac{t^\alpha}{\Gamma(1+\alpha)},
\qquad
X_i^\alpha(x):=\frac{x_i^\alpha}{\Gamma(1+\alpha)},
\quad i=1,2,3,
\end{equation}
and the polynomial coefficient algebra
\begin{equation}
\mathscr P_\alpha(U_{\mathrm{st}})
:=\R[\Theta^\alpha,X_1^\alpha,X_2^\alpha,X_3^\alpha].
\end{equation}
The polynomial operators in the four coordinate directions are denoted by
\begin{equation}
\mathsf D_t^\alpha,
\qquad
\mathsf D_i^\alpha\quad(i=1,2,3),
\end{equation}
with \(\mathsf D_t^\alpha\) acting on the generator \(\Theta^\alpha\) and
\(\mathsf D_i^\alpha\) acting on \(X_i^\alpha\), according to
Definition~\ref{def:poly-caputo-operator}.

We write
\begin{equation}
e^0:=d^\alpha\Theta,
\qquad
e^i:=d^\alpha X_i,
\quad i=1,2,3,
\end{equation}
for the distinguished coframe generators. For multi-indices we use the shorthand
\begin{equation}
e^{\mu\nu}:=e^\mu\wedge e^\nu,
\qquad
e^{\mu\nu\lambda}:=e^\mu\wedge e^\nu\wedge e^\lambda,
\qquad
e^{0123}:=e^0\wedge e^1\wedge e^2\wedge e^3.
\end{equation}
For a scalar coefficient \(f\) and a spatial vector field \(V=(V_1,V_2,V_3)\), define
\begin{equation}
\nabla_\alpha f
:=
(\mathsf D_1^\alpha f,\mathsf D_2^\alpha f,\mathsf D_3^\alpha f),
\end{equation}
\begin{equation}
\operatorname{div}_\alpha V
:=
\mathsf D_1^\alpha V_1
+
\mathsf D_2^\alpha V_2
+
\mathsf D_3^\alpha V_3,
\end{equation}
and
\begin{equation}
\operatorname{curl}_\alpha V
:=
\bigl(
\mathsf D_2^\alpha V_3-\mathsf D_3^\alpha V_2,\,
\mathsf D_3^\alpha V_1-\mathsf D_1^\alpha V_3,\,
\mathsf D_1^\alpha V_2-\mathsf D_2^\alpha V_1
\bigr).
\end{equation}
We also write
\begin{equation}
\Delta_\alpha V
:=
\bigl(
\sum_{j=1}^3\mathsf D_j^\alpha\mathsf D_j^\alpha V_1,\,
\sum_{j=1}^3\mathsf D_j^\alpha\mathsf D_j^\alpha V_2,\,
\sum_{j=1}^3\mathsf D_j^\alpha\mathsf D_j^\alpha V_3
\bigr)
\end{equation}
and
\begin{equation}
\square_\alpha:=\mathsf D_t^\alpha\mathsf D_t^\alpha-\Delta_\alpha.
\end{equation}
Here \(\mathsf D_t^\alpha\mathsf D_t^\alpha\) denotes repeated application of the
order-\(\alpha\) polynomial operator on the coefficient algebra; it is not
identified with a single Caputo derivative of order \(2\alpha\).

\subsection{Faraday and current forms}

Let \(E=(E_1,E_2,E_3)\), \(B=(B_1,B_2,B_3)\), \(\rho\), and
\(J=(J_1,J_2,J_3)\) have coefficients in
\(\mathscr P_\alpha(U_{\mathrm{st}})\). Define the Faraday \(2\)-form by
\begin{equation}
F
:=
- E_1e^{01}-E_2e^{02}-E_3e^{03}
+ B_1e^{23}-B_2e^{13}+B_3e^{12},
\end{equation}
and define the current \(3\)-form by
\begin{equation}
\mathcal J
:=
\rho\,e^{123}
-
J_1\,e^{023}
+
J_2\,e^{013}
-
J_3\,e^{012}.
\end{equation}

\begin{proposition}[homogeneous fractional Maxwell equations]
\label{prop:maxwell-homogeneous-new}
For the \(2\)-form \(F\), one has
\begin{equation}
d^\alpha F=0
\end{equation}
if and only if
\begin{equation}
\operatorname{div}_\alpha B=0,
\qquad
\mathsf D_t^\alpha B+\operatorname{curl}_\alpha E=0.
\end{equation}
\end{proposition}

\begin{proof}
Since the generators \(e^\mu\) are constant in the exterior algebra, the operator
\(d^\alpha\) acts only on the coefficients. Expanding
Definition~\ref{def:dalpha-new} and collecting the coefficients of the ordered
basis \(3\)-forms
\begin{equation}
e^{123},\qquad e^{023},\qquad e^{013},\qquad e^{012},
\end{equation}
yields
\begin{equation}
d^\alpha F
=
(\operatorname{div}_\alpha B)e^{123}
+
\bigl(\mathsf D_t^\alpha B_1+(\operatorname{curl}_\alpha E)_1\bigr)e^{023}
\end{equation}
\begin{equation}
\qquad
-
\bigl(\mathsf D_t^\alpha B_2+(\operatorname{curl}_\alpha E)_2\bigr)e^{013}
+
\bigl(\mathsf D_t^\alpha B_3+(\operatorname{curl}_\alpha E)_3\bigr)e^{012}.
\end{equation}
Since these basis \(3\)-forms are linearly independent, the claim follows.
\end{proof}

\begin{proposition}[continuity form]
\label{prop:continuity-form-new}
For the \(3\)-form \(\mathcal J\), one has
\begin{equation}
d^\alpha\mathcal J=0
\end{equation}
if and only if
\begin{equation}
\mathsf D_t^\alpha \rho+\operatorname{div}_\alpha J=0.
\end{equation}
\end{proposition}

\begin{proof}
Applying Definition~\ref{def:dalpha-new} and collecting the coefficient of
\(e^{0123}\), we obtain
\begin{equation}
d^\alpha\mathcal J
=
\bigl(
\mathsf D_t^\alpha\rho
+
\mathsf D_1^\alpha J_1
+
\mathsf D_2^\alpha J_2
+
\mathsf D_3^\alpha J_3
\bigr)e^{0123}.
\end{equation}
Hence \(d^\alpha\mathcal J=0\) is equivalent to the stated continuity equation.
\end{proof}

\subsection{A Lorentzian coframe-attached Hodge operator}

The previous constructions are independent of a metric. To formulate a
Maxwell-type system with the usual hyperbolic sign structure, we now attach a
Lorentzian sign convention to the distinguished coframe
\begin{equation}
e^0=d^\alpha\Theta,
\qquad
e^i=d^\alpha X_i,
\quad i=1,2,3.
\end{equation}
Define
\begin{equation}
\star_\alpha:
\Omega_\alpha^2(U_{\mathrm{st}};\mathscr P_\alpha)
\to
\Omega_\alpha^2(U_{\mathrm{st}};\mathscr P_\alpha)
\end{equation}
as the \(\mathscr P_\alpha(U_{\mathrm{st}})\)-linear map determined by
\begin{equation}
\star_\alpha(e^{01})=-e^{23},
\qquad
\star_\alpha(e^{02})=e^{13},
\qquad
\star_\alpha(e^{03})=-e^{12},
\end{equation}
\begin{equation}
\star_\alpha(e^{23})=e^{01},
\qquad
\star_\alpha(e^{13})=-e^{02},
\qquad
\star_\alpha(e^{12})=e^{03}.
\end{equation}
Then
\begin{equation}
\star_\alpha^2=-\mathrm{id}
\qquad\text{on }
\Omega_\alpha^2(U_{\mathrm{st}};\mathscr P_\alpha),
\end{equation}
as in the Lorentzian four-dimensional Hodge theory on \(2\)-forms.

We note that the operator \(\star_\alpha\) is coframe-attached: it is defined with respect to
the distinguished fractional coframe
\((d^\alpha\Theta,d^\alpha X_1,d^\alpha X_2,d^\alpha X_3)\). It is not claimed to
come from a fully intrinsic fractional Lorentzian metric. Its purpose is to
encode the Lorentzian sign convention needed for the Maxwell-type equations
below.

\subsection{The Lorentzian coframe-attached Maxwell system}

\begin{definition}[Lorentzian coframe-attached fractional Maxwell system]
\label{def:maxwell-type-system}
\label{def:lorentzian-fractional-maxwell}
We say that \((F,\mathcal J)\) satisfies the Lorentzian coframe-attached
fractional Maxwell system if
\begin{equation}
d^\alpha F=0,
\qquad
d^\alpha(\star_\alpha F)=\mathcal J.
\end{equation}
\end{definition}

\begin{proposition}[component form of the Maxwell system]
\label{prop:component-maxwell-lorentzian}
The system in Definition~\ref{def:lorentzian-fractional-maxwell} is equivalent to
\begin{equation}
\operatorname{div}_\alpha B=0,
\qquad
\mathsf D_t^\alpha B+\operatorname{curl}_\alpha E=0,
\end{equation}
and
\begin{equation}
\operatorname{div}_\alpha E=\rho,
\qquad
\operatorname{curl}_\alpha B-\mathsf D_t^\alpha E=J.
\end{equation}
\end{proposition}

\begin{proof}
The homogeneous equation \(d^\alpha F=0\) is exactly
Proposition~\ref{prop:maxwell-homogeneous-new}. For the inhomogeneous equation,
the Lorentzian coframe-attached Hodge operator gives
\begin{equation}
\star_\alpha F
=
E_1e^{23}-E_2e^{13}+E_3e^{12}
+B_1e^{01}+B_2e^{02}+B_3e^{03}.
\end{equation}
Collecting coefficients in \(d^\alpha(\star_\alpha F)\), we find
\begin{equation}
d^\alpha(\star_\alpha F)
=
(\operatorname{div}_\alpha E)e^{123}
-
\bigl((\operatorname{curl}_\alpha B)_1-\mathsf D_t^\alpha E_1\bigr)e^{023}
\end{equation}
\begin{equation}
\qquad
+
\bigl((\operatorname{curl}_\alpha B)_2-\mathsf D_t^\alpha E_2\bigr)e^{013}
-
\bigl((\operatorname{curl}_\alpha B)_3-\mathsf D_t^\alpha E_3\bigr)e^{012}.
\end{equation}
Equating this with
\begin{equation}
\mathcal J
=
\rho e^{123}-J_1e^{023}+J_2e^{013}-J_3e^{012}
\end{equation}
gives precisely
\begin{equation}
\operatorname{div}_\alpha E=\rho,
\qquad
\operatorname{curl}_\alpha B-\mathsf D_t^\alpha E=J.
\end{equation}
\end{proof}

\begin{theorem}[fractional charge conservation]
\label{thm:lorentzian-charge-conservation}
\label{thm:charge-conservation-new}
Every solution of the Lorentzian coframe-attached fractional Maxwell system
satisfies
\begin{equation}
d^\alpha\mathcal J=0.
\end{equation}
Equivalently,
\begin{equation}
\mathsf D_t^\alpha\rho+\operatorname{div}_\alpha J=0.
\end{equation}
\end{theorem}

\begin{proof}
By Definition~\ref{def:lorentzian-fractional-maxwell},
\begin{equation}
\mathcal J=d^\alpha(\star_\alpha F).
\end{equation}
Therefore, by nilpotency,
\begin{equation}
d^\alpha\mathcal J=(d^\alpha)^2(\star_\alpha F)=0.
\end{equation}
The component form is Proposition~\ref{prop:continuity-form-new}.
\end{proof}

\begin{lemma}[spacetime commutation on polynomial coefficients]
\label{lem:spacetime-poly-commutation}
On \(\mathscr P_\alpha(U_{\mathrm{st}})\), the polynomial Caputo operators commute
in all spacetime directions:
\begin{equation}
\mathsf D_\mu^\alpha\mathsf D_\nu^\alpha
=
\mathsf D_\nu^\alpha\mathsf D_\mu^\alpha
\qquad(\mu,\nu=0,1,2,3),
\end{equation}
where \(\mathsf D_0^\alpha=\mathsf D_t^\alpha\). In particular,
\begin{equation}
\mathsf D_t^\alpha(\operatorname{curl}_\alpha V)
=
\operatorname{curl}_\alpha(\mathsf D_t^\alpha V)
\end{equation}
for every polynomial vector field \(V\).
\end{lemma}

\begin{proof}
The first assertion is Lemma~\ref{lem:poly-stability-commutation} applied to the
spacetime polynomial algebra
\begin{equation}
\mathscr P_\alpha(U_{\mathrm{st}})
=
\R[\Theta^\alpha,X_1^\alpha,X_2^\alpha,X_3^\alpha].
\end{equation}
The displayed commutation with \(\operatorname{curl}_\alpha\) follows component
by component, since each component of \(\operatorname{curl}_\alpha V\) is a
signed sum of spatial polynomial Caputo derivatives of the components of \(V\).
\end{proof}

\begin{lemma}[fractional vector identities on \(\mathscr P_\alpha\)]
\label{lem:fractional-vector-identities}
On the polynomial coefficient algebra \(\mathscr P_\alpha(U_{\mathrm{st}})\), one has
\begin{equation}
\operatorname{div}_\alpha(\operatorname{curl}_\alpha V)=0,
\qquad
\operatorname{curl}_\alpha(\nabla_\alpha f)=0,
\end{equation}
and
\begin{equation}
\operatorname{curl}_\alpha(\operatorname{curl}_\alpha V)
=
\nabla_\alpha(\operatorname{div}_\alpha V)-\Delta_\alpha V.
\end{equation}
\end{lemma}

\begin{proof}
The proof is the same algebraic cancellation as in the classical case, with
ordinary partial derivatives replaced by the commuting polynomial operators
\(\mathsf D_i^\alpha\). The required commutation of mixed spatial derivatives is
Lemma~\ref{lem:spacetime-poly-commutation}. Expanding the components gives, for
example,
\begin{equation}
\operatorname{div}_\alpha(\operatorname{curl}_\alpha V)
=
\mathsf D_1^\alpha(\mathsf D_2^\alpha V_3-\mathsf D_3^\alpha V_2)
+\mathsf D_2^\alpha(\mathsf D_3^\alpha V_1-\mathsf D_1^\alpha V_3)
+\mathsf D_3^\alpha(\mathsf D_1^\alpha V_2-\mathsf D_2^\alpha V_1)=0,
\end{equation}
with cancellation following from pairwise commutation. The identity
\(\operatorname{curl}_\alpha(\nabla_\alpha f)=0\) is identical componentwise. For
the curl--curl identity, the first component is
\begin{equation}
(\operatorname{curl}_\alpha\operatorname{curl}_\alpha V)_1
=
\mathsf D_2^\alpha(\mathsf D_1^\alpha V_2-\mathsf D_2^\alpha V_1)
-
\mathsf D_3^\alpha(\mathsf D_3^\alpha V_1-\mathsf D_1^\alpha V_3),
\end{equation}
which, using commutation, equals
\begin{equation}
\mathsf D_1^\alpha(\mathsf D_2^\alpha V_2+\mathsf D_3^\alpha V_3)
-
\bigl(\mathsf D_2^\alpha\mathsf D_2^\alpha V_1+
\mathsf D_3^\alpha\mathsf D_3^\alpha V_1\bigr).
\end{equation}
Adding and subtracting \(\mathsf D_1^\alpha\mathsf D_1^\alpha V_1\) gives the
first component of
\(\nabla_\alpha(\operatorname{div}_\alpha V)-\Delta_\alpha V\). The other two
components are obtained by the same cyclic expansion.
\end{proof}

\begin{theorem}[fractional wave equations]
\label{thm:fractional-wave-equations}
Every solution of the Lorentzian coframe-attached fractional Maxwell system
satisfies
\begin{equation}
\square_\alpha B
=
\operatorname{curl}_\alpha J,
\end{equation}
and
\begin{equation}
\square_\alpha E
=
-\mathsf D_t^\alpha J-\nabla_\alpha\rho.
\end{equation}
In particular, in the source-free case \(\rho=0\) and \(J=0\), both \(E\) and
\(B\) satisfy
\begin{equation}
\square_\alpha E=0,
\qquad
\square_\alpha B=0.
\end{equation}
\end{theorem}

\begin{proof}
From the inhomogeneous equation,
\begin{equation}
\operatorname{curl}_\alpha B-\mathsf D_t^\alpha E=J.
\end{equation}
Taking \(\operatorname{curl}_\alpha\) and using
Lemma~\ref{lem:spacetime-poly-commutation} gives
\begin{equation}
\operatorname{curl}_\alpha\operatorname{curl}_\alpha B
-
\mathsf D_t^\alpha(\operatorname{curl}_\alpha E)
=
\operatorname{curl}_\alpha J.
\end{equation}
Using the homogeneous equation
\begin{equation}
\operatorname{curl}_\alpha E=-\mathsf D_t^\alpha B
\end{equation}
and again using Lemma~\ref{lem:spacetime-poly-commutation}, we obtain
\begin{equation}
\operatorname{curl}_\alpha\operatorname{curl}_\alpha B
+
\mathsf D_t^\alpha\mathsf D_t^\alpha B
=
\operatorname{curl}_\alpha J.
\end{equation}
By Lemma~\ref{lem:fractional-vector-identities} and
\(\operatorname{div}_\alpha B=0\),
\begin{equation}
\operatorname{curl}_\alpha\operatorname{curl}_\alpha B
=
-\Delta_\alpha B.
\end{equation}
Therefore
\begin{equation}
\mathsf D_t^\alpha\mathsf D_t^\alpha B-
\Delta_\alpha B
=
\operatorname{curl}_\alpha J,
\end{equation}
which is the asserted equation for \(B\).

For \(E\), apply \(\mathsf D_t^\alpha\) to
\begin{equation}
\operatorname{curl}_\alpha B-\mathsf D_t^\alpha E=J.
\end{equation}
This gives
\begin{equation}
\mathsf D_t^\alpha(\operatorname{curl}_\alpha B)
-
\mathsf D_t^\alpha\mathsf D_t^\alpha E
=
\mathsf D_t^\alpha J.
\end{equation}
By Lemma~\ref{lem:spacetime-poly-commutation} and the homogeneous equation
\begin{equation}
\mathsf D_t^\alpha B=-\operatorname{curl}_\alpha E,
\end{equation}
we get
\begin{equation}
-\operatorname{curl}_\alpha\operatorname{curl}_\alpha E
-
\mathsf D_t^\alpha\mathsf D_t^\alpha E
=
\mathsf D_t^\alpha J.
\end{equation}
Using
\begin{equation}
\operatorname{curl}_\alpha\operatorname{curl}_\alpha E
=
\nabla_\alpha(\operatorname{div}_\alpha E)-\Delta_\alpha E
=
\nabla_\alpha\rho-\Delta_\alpha E,
\end{equation}
we obtain
\begin{equation}
\mathsf D_t^\alpha\mathsf D_t^\alpha E-
\Delta_\alpha E
=
-\mathsf D_t^\alpha J-\nabla_\alpha\rho.
\end{equation}
This proves the claim.
\end{proof}

\begin{remark}
The vector identities, charge conservation law, and wave equations in this
section are algebraic consequences of the polynomial coefficient calculus. Under
the diagonal chain isomorphism \(\Phi\) of Proposition~\ref{prop:dRham-reduction},
applied with \(d=4\) to the spacetime variables
\((\Theta^\alpha,X_1^\alpha,X_2^\alpha,X_3^\alpha)\),
these identities correspond to the classical polynomial Maxwell identities in
the associated \(y\)-coordinates. Thus the fractional features of this section
reside in the chosen anchored coframe and in the polynomial Caputo operators
attached to it, not in new algebraic identities beyond their classical
\(\Phi\)-conjugates. In particular, the fractional wave operator is built from
commuting polynomial Caputo operators on \(\mathscr P_\alpha(U_{\mathrm{st}})\),
and repeated order-\(\alpha\) time differentiation is treated as an iterated
operator rather than collapsed into a single derivative of order \(2\alpha\).
\end{remark}

\section{Concluding remarks}

We have developed a coordinate-anchored pointwise theory of Caputo-type
fractional tangent functionals. The central local result is the representation
\begin{equation}
\Dcal_i^\alpha(p)=\R\cdot \bigl((\partial_{C,i}^\alpha)|_p\bigr),
\qquad
T_p^\alpha U=\bigoplus_{i=1}^d\Dcal_i^\alpha(p),
\end{equation}
obtained by working on the range of the coordinatewise
Riemann--Liouville fractional integral and defining the Caputo operator by exact
inversion. This also separates the normalized slice from the ambient tangent
vector space, and shows that at interior base points the resulting tangent space
acts faithfully on the common anchored space \(\Acal^\alpha(U)\).

The formal exterior calculus constructed from the symbols \(d^\alpha X_i\) is a
global algebraic realization of this local picture. On the polynomial coefficient
algebra generated by the fractional coordinate primitives, \(d^\alpha\) is stable
and nilpotent, admits an explicit fractional polynomial Poincar\'e homotopy, and
supports a Lorentzian coframe-attached Maxwell-type system with charge
conservation and fractional wave equations. The coordinate-rigidity theorem
shows that, among real-valued positive-cone-preserving linear coordinate changes,
the associated fractional coframe is preserved only by permutations and positive
rescalings of the coordinates.

Furthermore, the polynomial
complex is chain-isomorphic to the ordinary polynomial de~Rham complex by the
diagonal rescaling of Proposition~\ref{prop:dRham-reduction}, so it does not
produce new polynomial cohomology. The construction's value is therefore structural: it shows that
the anchored local tangent--cotangent theory admits a stable nilpotent exterior
complex and a mathematically consistent differential-form PDE model. 

Natural follow-up problems are therefore to identify invariant formulations of
the anchored construction, to enlarge the coefficient algebra while preserving
nilpotency, and to develop analytic or numerical theories for the Maxwell-type
system. In particular, the nonlocality of the Caputo operator obstructs a direct
elementwise finite element exterior calculus as developed in \cite{arnold2010finite,ArnoldFalkWinther2006}, but the polynomial subcomplexes
\(\Omega_{\alpha,\le N}^\bullet\) provide a natural global Galerkin starting
point. For models with fractional time and classical space, hybrid schemes
combining finite element exterior calculus in space with Caputo-in-time
discretizations may be the most practical route.

\bibliographystyle{abbrv}
\small \setlength{\bibsep}{0.4pt}
\bibliography{lit}

\end{document}